\documentclass[11pt,leqno]{article}
\usepackage{amsmath}
\usepackage{amssymb}
\usepackage{theorem}
\newcommand{\To}{\ensuremath{\rightrightarrows}}

\newcommand{\fenv}[1]%
{\ensuremath{\,\overrightarrow{\operatorname{env}}_{#1}}}
\newcommand{\benv}[1]%
{\ensuremath{\,\overleftarrow{\operatorname{env}}_{#1}}}
\newcommand{\emp}{\ensuremath{\varnothing}}

\newcommand{\scal}[2]{\left\langle{#1},{#2}  \right\rangle}

\newcommand{\RR}{\ensuremath{\mathbb R}}

\newcommand{\RPX}{\ensuremath{\left[0,+\infty\right]}}

\newcommand{\RX}{\ensuremath{\,\left]-\infty,+\infty\right]}}

\newcommand{\ball}{\ensuremath{B}}

\newcommand{\Hess}{\ensuremath{\nabla^2\!}}
\newcommand{\oldIDD}{\ensuremath{\operatorname{int}\operatorname{dom}f}}
\newcommand{\IDD}{\ensuremath{U}}

\newcommand{\dom}{\ensuremath{\operatorname{dom}}}
\newcommand{\argmin}{\ensuremath{\operatorname*{argmin}}}

\newcommand{\intdom}{\ensuremath{\operatorname{int}\operatorname{dom}}\,}

\newcommand{\inte}{\ensuremath{\operatorname{int}}}
\newcommand{\deriv}{\ensuremath{\operatorname{\; d}}}
\newcommand{\rockderiv}{\ensuremath{\operatorname{\; \hat{d}}}}
\newcommand{\closu}{\ensuremath{\operatorname{cl}}}

\newcommand{\ran}{\ensuremath{\operatorname{ran}}}
\newcommand{\conv}{\ensuremath{\operatorname{conv}}}

\newcommand{\Id}{\ensuremath{\operatorname{Id}}}

\newcommand{\fproj}[1]{\overrightarrow{P\thinspace}_%
{\negthinspace\negthinspace #1}}

\newcommand{\bproj}[1]{\overleftarrow{\thinspace P\thinspace}_%
{\negthinspace\negthinspace #1}}

\newcommand{\proj}[1]{{\thinspace P\thinspace}_%
{\negthinspace\negthinspace #1}}
\newcommand{\fD}[1]{\overrightarrow{D\thinspace}_%
{\negthinspace\negthinspace #1}}

\newcommand{\bD}[1]{\overleftarrow{\thinspace D\thinspace}_%
{\negthinspace\negthinspace #1}}

\newcommand{\bDC}{$\bD{}\,$-Chebyshev}

\newcommand{\pinf}{\ensuremath{+\infty}}

{\begin{list}{}{%
\settowidth{\labelwidth}{\textrm{#1~}}%
\setlength{\leftmargin}{\labelwidth+\labelsep}}}%requires macro calc.sty
{\end{list}}
\newtheorem{theorem}{Theorem}[section]
\newtheorem{lemma}[theorem]{Lemma}
\newtheorem{corollary}[theorem]{Corollary}
\newtheorem{proposition}[theorem]{Proposition}
\newtheorem{definition}[theorem]{Definition}
\theoremstyle{plain}{\theorembodyfont{\rmfamily}
}
\theoremstyle{plain}{\theorembodyfont{\rmfamily}
}
\theoremstyle{plain}{\theorembodyfont{\rmfamily}
}
\theoremstyle{plain}{\theorembodyfont{\rmfamily}
\newtheorem{example}[theorem]{Example}}
\newtheorem{fact}[theorem]{Fact}
\theoremstyle{plain}{\theorembodyfont{\rmfamily}
\newtheorem{remark}[theorem]{Remark}}

\begin{document}

\title{\sffamily Bregman distances and Chebyshev sets}

\author{Heinz H.\ Bauschke\thanks{Mathematics, Irving K.\ Barber School,
The University of British Columbia Okanagan, Kelowna,
B.C. V1V 1V7, Canada. Email:
\texttt{heinz.bauschke@ubc.ca}.},~
Xianfu
Wang\thanks{Mathematics, Irving K.\ Barber School,
The University of British Columbia Okanagan, Kelowna,
B.C. V1V 1V7, Canada.
Email:
\texttt{shawn.wang@ubc.ca}.},~
Jane Ye\thanks{Department of Mathematics and Statistics, University of
Victoria, Victoria, B.C. V8W 3P4, Canada.
Email:~\texttt{janeye@math.uvic.ca}.},~
and~Xiaoming Yuan\thanks{Department of Management Science,
Antai College of Economics and Management, 
Shanghai Jiao Tong University, Shanghai, 200052, China.
Email: \texttt{xmyuan@sjtu.edu.cn}.}}

\date{December 24, 2007}

\maketitle

\vskip 8mm

\begin{abstract} \noindent
A closed set of a Euclidean space is said to be
Chebyshev if every point in the space  has one and only one closest
point in the set. Although the situation is not settled
in infinite-dimensional Hilbert spaces, in
1932 Bunt showed that in Euclidean spaces a
closed set is Chebyshev if and only if the set is convex. In this
paper, from the more general perspective of Bregman
distances, we show that if every point in the space has a unique
nearest point in a closed set, then the set is convex. We provide
two approaches: one is by nonsmooth analysis; the other by maximal
monotone operator theory. Subdifferentiability properties of Bregman nearest
distance functions are also given.
\end{abstract}

{\small

}

\section{Introduction}

Throughout, $\RR^J$ is the standard Euclidean space with inner
product $\scal{\cdot}{\cdot}$ and induced norm $\|\cdot\|$,
and $\Gamma$ is the set of proper lower semicontinuous convex functions on $\RR^J$. Let $C$
be a nonempty closed subset of $\RR^J$. If each $x\in \RR^J$
has a unique nearest point in $C$, the set $C$ is called
Chebyshev. The famous Chebyshev set problem inquires: ``Is a Chebyshev set
necessarily convex?''. It has been studied by many authors, see
\cite{Edgar,jborwein,Frank1,deutsch,lewis,urruty1,urruty3,urruty4}
and the references therein.
Although answered in the affirmative by Bunt in 1932,
we look at the problem from the more general point of view
of Bregman distances.

Let
\begin{equation} \label{eq:preD}
\text{$f \colon \RR^J\to\RX$ be convex and differentiable on $U :=
\oldIDD \neq \emp$.}
\end{equation}
The \emph{Bregman distance} associated with $f$ is defined by
\begin{equation} \label{eq:D}
D\colon \RR^J \times \RR^J \to \RPX \colon (x,y) \mapsto
\begin{cases}
f(x)-f(y)-\scal{\nabla f(y)}{x-y}, &\text{if}\;\;y\in\IDD;\\
\pinf, & \text{otherwise}.
\end{cases}
\end{equation}
Assume that $C\subset U$.  It is a natural generalization of the
Chebyshev problem to ask the following:
\begin{quotation}
\noindent
``If every $x\in U$ has ---  in terms of the Bregman distance --- a unique
nearest point in $C$, i.e., $C$ is Chebyshev for the Bregman distance, must $C$
be convex?''
\end{quotation}
We give two approaches to our affirmative answer:
one uses beautiful properties of maximal monotone
operators: Rockafellar's virtual convexity theorem on ranges of
maximal monotone operators; the other uses generalized
subdifferentials from nonsmooth analysis, which allows us to
characterize Chebyshev sets. We also study subdifferentiabilities of
Bregman distance functions associated to closed sets. These
nonsmooth analysis results are interesting in their own right, since
Bregman distances have found tremendous applications
in Statistics, Engineering, and Optimization; see the recent books
\cite{ButIus,CenZen} and the references therein.

The function $D$ does not define a metric, since it is not symmetric and
does not satisfy the triangle inequality.
It is thus remarkable that it is not only possible to derive many results on projections and distances similar to the one obtained in finite
dimensional Euclidean spaces, but also to provide a general framework for
best approximations.

The paper is organized as follows. In Section~\ref{assumption}, we state
our assumptions on $f$ and provide some concrete choices. In Section~\ref{geodesicscurve}, we characterize left Bregman nearest points and geodesics. We show that the Bregman normal is a proximal normal. In Section~\ref{monotone}, when $f$ is Legendre and $1$-coercive and $C$ is Chebyshev, we show that the composition of the Bregman nearest-point map and $\nabla f^*$ is maximal monotone. This allows us to apply Rockafellar's theorem on virtual convexity of range of maximal monotone operator to obtain that a Chebyshev set is convex.
In Section~\ref{clarkemor}, we study subdifferentiability properties of
left Bregman distance function. Formulas for the Clarke subdifferential,
the limiting subdifferential and the Dini subdifferential are given. In
Section~\ref{completecheby}, we give a complete characterizations of
Chebyshev sets. Our approach generalizes the results given by
Hiriart-Urruty \cite{urruty4,urruty1} from the Euclidean to the Bregman setting. Finally, in Section~\ref{rightchar} we show that the convexity of Chebyshev sets for right Bregman projections of $f$ can be studied by using the left Bregman projections of $f^*$. We give an example showing
that even if the right Bregman projection is single-valued, the set $C$ 
need not be convex.

\noindent{\bf Notation:} In $\RR^J$, the closed ball centered at $x$ with
radius $\delta>0$ is denoted by ${B}_{\delta}(x)$ and the closed unit ball
is $\ball = B_1(0)$. For a set $S$, the expressions
$\inte S$, $\closu S$, $\conv S$ signify the interior, closure, and convex hull of $S$ respectively.
For a set-valued mapping $T:\RR^J\To\RR^J$, we use $\ran T$ and $\dom T$ for its range and domain, and
$T^{-1}$ for its set-valued inverse, i.e.,
$x\in T^{-1}(y) \Leftrightarrow y\in T(x).$ 
For a function $f:\RR^J\rightarrow \RX$, $\dom f$ is the domain of $f$, and $f^*$ is its Fenchel conjugate;
$\conv f$ ($\closu\conv f$) denotes the convex hull (closed convex hull) of $f$. For a differentiable function $f$,
$\nabla f(x)$ and $\nabla^2 f(x)$ denote the gradient vector and the Hessian matrix at $x$.
Our notation is standard and follows, e.g., \cite{Rock70,Rock98}.

\section{Standing Assumptions and Examples}\label{assumption}
From now on, and until the end of Section~\ref{completecheby},
our standing assumptions on $f$ and $C$ are:
\begin{itemize}
\item[\bfseries A1]
$f\in\Gamma$ is a convex function of Legendre type,
i.e., $f$ is essentially smooth and essentially strictly convex in the
sense of \cite[Section~26]{Rock70}.
\item[\bfseries A2]
$f$ is $1$-coercive, i.e., $\displaystyle \lim_{\|x\|\rightarrow
+\infty}f(x)/\|x\|=+\infty$. An equivalent requirement is $\dom f^*=\RR^J$
(see \cite[Theorem~11.8(d)]{Rock98}).
\item[\bfseries A3] The set $C$ is a nonempty closed subset of $\IDD$.
\end{itemize}
Important instances of functions satisfying the above conditions are:
\begin{example}
Let $x=(x_j)_{1\leq j\leq J}$ and $y=(y_j)_{1\leq j\leq J}$ be
two points in $\RR^J$.
\begin{enumerate}
\item
\emph{Energy:}
If $f\colon x\mapsto\tfrac{1}{2}\|x\|^2$, then $U=\RR^J$,
\begin{equation*}
D(x,y) = \frac{1}{2}\|x-y\|^2,
\end{equation*}
and $\Hess f(x)=\Id$ for every $x\in\RR^J$. Note that $f^*(x)=\frac{1}{2}\|x\|^2$, $\dom f^*=\RR^J$,
and $\Hess f^*=\Id$.
\item  \label{ex:examples:KL}
\emph{Boltzmann-Shannon entropy:}
If $f\colon x\mapsto\sum_{j=1}^{J}x_j\ln(x_j)-x_j$ if $x\geq 0$, $+\infty$ otherwise.
Here $x\geq 0$ means $x_{j}\geq 0$ for $1\leq j\leq J$ and
similarly for $x> 0$, and $0\ln 0=0$.
Then $U=\{x\in \RR^J\colon x > 0\}$, and
\begin{equation*}
D(x,y) = \begin{cases}
\textstyle \sum_{j=1}^J x_j \ln(x_j/y_j) - x_j + y_j, &
\text{if $x \geq 0$ and $y>0$;}\\
\pinf, & \text{otherwise}
\end{cases}
\end{equation*}
is the so-called \emph{Kullback-Leibler divergence}.
Note that
$$\Hess f(x)= \begin{pmatrix}
1/x_{1} &  0 & \cdots & 0& 0\\
0 & 1/x_{2} & 0& \cdots    & 0\\
\vdots & 0 & \ddots &0 & 0 \\
0 & \ldots & 0 & 1/x_{J-1}& 0\\
0 & \ldots &0 & 0 & 1/x_{J}
\end{pmatrix},
$$
that
$ f^*(x)=\sum_{j=1}^{J}e^{x_{j}}$ with $\dom f^*=\RR^{J}$, and
that $$\Hess f^*(x)=
\begin{pmatrix}
e^{x_{1}} &  0 & \cdots & 0& 0\\
0 & e^{x_{2}} & 0& \cdots    & 0\\
\vdots & 0 & \ddots & 0& 0 \\
0 & \ldots & 0 & e^{x_{J-1}}& 0\\
0 & \ldots &0 & 0 & e^{x_{J}}
\end{pmatrix}.
$$
\item \emph{Fermi-Dirac entropy}: If $f:x\mapsto \sum_{j=1}^{J}x_{j}\ln x_{j}+(1-x_{j})\ln (1-x_{j})$,
then $U=\{x\in \RR^J: 0<x<1\}$ and
$$
D(x,y) = \begin{cases}
\textstyle \sum_{j=1}^J x_j \ln(x_j/y_j)+(1-x_{j})\ln((1-x_{j})/(1-y_{j})), &
\text{if $1\geq x \geq 0$ and $1>y>0$;}\\
\pinf, & \text{otherwise.}
\end{cases}
$$
While
$$\Hess f(x)=
\begin{pmatrix}
\frac{1}{x_{1}(1-x_{1})} &  0 & \cdots & 0\\
0 & \frac{1}{x_{2}(1-x_{2})}& 0& 0 \\
\vdots & 0 & \ddots & 0\\
0 & \ldots & 0 & \frac{1}{x_{J}(1-x_{J})}\\
\end{pmatrix}, \quad \forall  0<x<1, x\in \RR^{J},
$$
we have
$f^*(x)=\sum_{j=1}^{J}\ln (1+e^{x_{j}})$ with
$$\Hess f^*(x)=
\begin{pmatrix}
\frac{e^{x_{1}}}{(1+e^{x_{1}})^2} &  0 & \cdots & 0\\
0 & \frac{e^{x_{2}}}{(1+e^{x_{2}})^2} & 0& 0\\
\vdots & 0 & \ddots & 0\\
0 & \ldots & 0 & \frac{e^{x_{J}}}{(1+e^{x_{J}})^2}\\
\end{pmatrix}, \quad \forall x\in \RR^J.
$$
\item In general, we can let
$f\colon x\mapsto\sum_{i=1}^J \phi(x_{i})$ where $\phi:\RR\rightarrow\RX$ is an Legendre function.
Then  $U=(\inte\dom\phi)^{J}$,
$$D(x,y)=\sum_{j=1}^{J}\phi(x_j)-\phi(y_{j})-\phi'(y_j)(x_j-y_j),\quad
\forall\; x\in\RR^J,y\in U.$$
In particular, one can use $\phi(t)=|t|^p/p$ with $p>1$.
\end{enumerate}
\end{example}
The following result (see \cite[Theorem~26.5]{Rock70}) plays
an important role in the sequel.
\begin{fact}[Rockafellar]\label{isom}
A convex function $f$ is of Legendre type if and only if $f^*$ is.
In this case, the gradient mapping
$$\nabla f:U \rightarrow\inte\dom f^*:x\mapsto \nabla f(x),$$
is a topological isomorphism with inverse mapping $(\nabla
f)^{-1}=\nabla f^*$.
\end{fact}

\section{Bregman Distances and Projection Operators}\label{geodesicscurve}
We start with
\begin{definition}
The \emph{left Bregman nearest-distance function} to $C$ is defined
by
\begin{equation}
\bD{C}\:  \colon U\to \RPX \colon y\mapsto
\inf_{x\in C}D(x,y),
\end{equation}
and the \emph{left Bregman nearest-point map} (i.e., the classical
Bregman projector) onto $C$ is
$$\bproj{C}\colon \IDD \to \IDD \colon y\mapsto
\underset{x\in C}{\operatorname{argmin}}\;\: D(x,y) = \{x\in C\colon D(x,y)
= \bD{C}(y)\}.$$
\end{definition}
The \emph{right Bregman distance} and \emph{right Bregman
projector} onto $C$ are defined analogously and denoted by
$\fD{C}$ and $\fproj{C}$, respectively.
Note that while in \cite{noll} the authors consider proximity operators associated with convex
set $C$, here our set $C$ need not be convex and we do not
assume that $D(\cdot,\cdot)$ is jointly convex.

We shall often need the following identity
\begin{equation}\label{3point}
D(c,y)-D(x,y)=f(c)-f(x)-\langle\nabla f(y),c-x\rangle,
\end{equation}
which is an immediate consequence of the definition.

Our first result characterizes the left Bregman nearest point.
\begin{proposition}\label{near}
Let $x\in C$ and $y\in U$.
\begin{enumerate}
\item[{\rm (i)}] Then
\begin{equation}\label{same}{x}\in \bproj{C}(y)\quad \Leftrightarrow \quad
D(c,x)\geq \langle \nabla f(y)-\nabla f(x), c-x\rangle \quad \forall c\in C.
\end{equation}
If $C$ is convex, then
\begin{equation}\label{convexpart} {x}\in
\bproj{C}(y)\quad \Leftrightarrow \quad 0\geq \langle \nabla
f(y)-\nabla f(x), c-x\rangle \quad \forall c\in C.
\end{equation}
\item[{\rm (ii)}]
Suppose that $x\in \bproj{C}(y)$. Then  the Bregman projection of
\begin{equation}\label{geodesics}
z_{\lambda}=\nabla f^*(\lambda \nabla f(y)+(1-\lambda)\nabla f(x)) \mbox{ with $0\leq \lambda< 1$,}
 \end{equation}
 on $C$ is singleton with
\begin{equation}\label{25:i}
\bproj{C}(z_{\lambda})=x.
\end{equation}
If $C$ is convex, \eqref{25:i} holds for every $\lambda\geq 0$.
\end{enumerate}
\end{proposition}

\begin{proof}
(i): By definition,
$x\in \bproj{C}(y)$ if and only if
$$0\leq  D(c,y)-D(x,y) \quad \forall c\in C;$$
equivalently,
$f(c)-f(x)\geq \langle \nabla f(y), c-x\rangle$ by (\ref{3point}).
Subtracting $\langle \nabla f(x), c-x\rangle$ from both sides,
we obtain
$$D(c,x)\geq \langle \nabla f(y)-\nabla f(x),c-x\rangle.$$
Hence (\ref{same}) holds.

The convex counterpart \eqref{convexpart} is well known and
follows, e.g.,  from \cite[Proposition~3.16]{Baus97}.

(ii): Assume that $x\in\bproj{C}(y)$ and $z_{\lambda}=\nabla f^*(\lambda \nabla f(y)+(1-\lambda)\nabla f(x))$
with $0\leq \lambda <1$.
Then by (\ref{same}),
\begin{equation}\label{charac}
D(c,x)\geq \langle \nabla f(y)-\nabla f(x), c-x\rangle \quad \forall c\in C.
\end{equation}
Take $c\in C$.
By Fact~\ref{isom}, $\nabla f\circ\nabla f^*=\Id$, we have
\begin{align}\label{relation}
&\langle\nabla f(z_{\lambda})-\nabla f(x),c-x\rangle\\
&= \langle\nabla f \circ\nabla f^*(\lambda\nabla f(y)+(1-\lambda)\nabla f(x))-\nabla f(x),c-x\rangle\\
&=\langle (\lambda\nabla f(y)+(1-\lambda)\nabla f(x))-\nabla f(x),c-x\rangle\\
&=\lambda \langle \nabla f(y)-\nabla f(x),c-x\rangle.
\end{align}
If $\langle\nabla f(y)-\nabla f(x),c-x\rangle\leq 0,$ then
$$\lambda \langle\nabla f(y)-\nabla f(x),c-x\rangle\leq 0\leq D(c,x);$$
if $\langle\nabla f(y)-\nabla f(x),c-x\rangle\geq 0$, then using $0\leq\lambda <1$ and (\ref{charac}),
$$\lambda \langle\nabla f(y)-\nabla f(x),c-x\rangle\leq\langle\nabla f(y)-\nabla f(x),c-x\rangle
\leq D(c,x).$$
In either case,  by (\ref{relation}) we have
$$\langle\nabla f(z_{\lambda})-\nabla f(x),c-x\rangle \leq D(c,x) \quad \forall c\in C.$$
Hence $x\in \bproj{C}(z_{\lambda})$ by (\ref{same}).
We proceed to show that $\bproj{C}(z_{\lambda})$ is a singleton. If $\lambda=0$, then $z_{\lambda}=x$, $\bproj{C}(x)=\{x\}$ by strict convexity of $f$. It remains to consider the case $0<\lambda<1$.
 Let
$\hat{x}\in \bproj{C}(z_{\lambda})$. Then
$D(x,z_{\lambda})=D(\hat{x},z_{\lambda})$, which is
$$f(\hat{x})-f(x)-\langle \nabla f(z_{\lambda}),\hat{x}-x\rangle =0,$$
by (\ref{3point}).
Using
$z_{\lambda}=\nabla f^*(\lambda \nabla f(y)+(1-\lambda)\nabla f(x)),$ we have
$$f(\hat{x})-f(x)-\langle \lambda \nabla f(y)+(1-\lambda)\nabla f(x),\hat{x}-x\rangle=0,$$
so that
$$\lambda [f(\hat{x})-f(x)-\langle \nabla
f(y),\hat{x}-x\rangle]+(1-\lambda)[f(\hat{x})-f(x)-\langle \nabla
f(x),\hat{x}-x\rangle]=0,
$$
and
$$\lambda [f(x)-f(\hat{x})-\langle \nabla f(y),x-\hat{x}\rangle]
=(1-\lambda)[f(\hat{x})-f(x)-\langle \nabla f(x),\hat{x}-x\rangle].
$$
This gives, by (\ref{3point}),
$\lambda [D(x,y)-D(\hat{x},y)]=(1-\lambda)D(\hat{x},x)$
and hence $$D(x,y)-D(\hat{x},y)=\frac{1-\lambda}{\lambda}D(\hat{x},x),$$
since $1>\lambda >0$.
If $\hat{x}\neq x$, then $D(\hat{x},x)>0$ by the strict convexity of $f$ so that $D(x,y)>D(\hat{x},y)$, and this
contradicts that $x\in \bproj{C}(y)$. Therefore, $\bproj{C}(z_{\lambda})=\{x\}$.

When $C$ is convex, by (\ref{convexpart}), $x\in \bproj{C}(y)$ if and only if
\begin{equation}\label{whythis}
\langle \nabla f(y)-\nabla f(x), c-x\rangle \leq 0, \quad \forall c\in C.
\end{equation}
If $z_{\lambda}=\nabla f^*(\lambda \nabla f(y)+(1-\lambda)\nabla f(x))$ with $\lambda\geq 0$, then
\begin{align}
&\langle \nabla f(z_{\lambda})-\nabla f(x), c-x\rangle=\langle \nabla f\circ\nabla f^*(\lambda \nabla f(y)+(1-\lambda)\nabla f(x))-\nabla f(x),c-x\rangle\\
& =\lambda\langle \nabla f(y)-\nabla f(x),c-x\rangle\leq 0.
\end{align}
By (\ref{convexpart}), $x\in\bproj{C}(z_{\lambda})$.  Applying (\ref{25:i}), we see that
$x=\bproj{C}(z_{\lambda})$. Indeed, select $\lambda_{1}>\lambda$. Since
\begin{equation}\label{ch1}
z_{\lambda}=\nabla f^*(\lambda \nabla f(y)+(1-\lambda)\nabla f(x))\quad  \Rightarrow \quad\nabla f(z_{\lambda})=
\nabla f(x)+\lambda (\nabla f(y)-\nabla f(x)),
\end{equation}
\begin{equation}\label{ch2}
z_{\lambda_{1}}=\nabla f^*(\lambda_{1} \nabla f(y)+(1-\lambda_{1})\nabla f(x))\quad \Rightarrow
\quad \nabla f(z_{\lambda_{1}})=\nabla f(x)+\lambda_{1} (\nabla f(y)-\nabla f(x)).
\end{equation}
Solve (\ref{ch2}) for $\nabla f(y)-\nabla f(x)$ and put into (\ref{ch1}) to get
$$\nabla f(z_{\lambda})
=\left(1-\frac{\lambda}{\lambda_{1}}\right)\nabla f(x)+\frac{\lambda}{\lambda_{1}}\nabla f(z_{\lambda_{1}}).$$ This gives
$$z_{\lambda}=\nabla f^*((1-\lambda/\lambda_{1})\nabla f(x)+\lambda/\lambda_{1}\nabla f(z_{\lambda_{1}})).$$
As $x\in \bproj{C}(z_{\lambda_{1}})$, (\ref{25:i}) applies.
\end{proof}

It is interesting to point out a connection to the \emph{proximal normal
cone} $N_{C}^{P}(x)$ of $C$ at $x\in C$
Recall
that $$N_{C}^{P}(x):=\{t(y-x):\; t\geq 0, x\in\proj{C}(y), y\in \RR^J\},$$
in which $\proj{C}$ denotes the usual projection on $C$ in terms of Euclidean norm, and each vector $t(y-x)$ is
called a proximal normal to $C$ at $x$;
see, e.g., \cite[Section~1.1]{Yuri} for further information.
\begin{proposition}
Suppose that $f$ is twice continuously differentiable on $U$,
let $y\in U$, and suppose that $x\in\bproj{C}(y)$.
Then $\nabla f(y)-\nabla f(x)\in N_{C}^{P}(x).$
\end{proposition}
\begin{proof}
By Proposition~\ref{near}(i),
\begin{equation}\label{proxnormal}
D(c,x)\geq \langle \nabla f(y)-\nabla f(x), c-x\rangle \quad \forall c\in C.
\end{equation}
Since the Hessian of $f$ is continuous, using Taylor's formula, we obtain
\begin{equation}\label{taylor}
D(c,x)=f(c)-f(x)-\langle\nabla f(x),c-x\rangle=\frac{1}{2}\langle c-x, \Hess f(\xi)(c-x)\rangle
\quad \mbox{ where $\xi\in [c,x]$}.
\end{equation}
Fix $\delta>0$. Since
$\Hess f$ is continuous on the compact set $C\cap {B}_{\delta}(x)$,
there exists $\sigma=\sigma(x,\delta) >0$
such that $\|\Hess f(\xi)\|\leq 2\sigma$ for every $\xi\in C\cap {B}_{\delta}(x)$. Then (\ref{taylor}) gives
$D(c,x)\leq\sigma \|c-x\|^2$. By (\ref{proxnormal}),
$$\sigma\|c-x\|^2\geq \langle \nabla f(y)-\nabla f(x), c-x\rangle \quad \forall c\in C\cap {B}_{\delta}(x).$$
By \cite[Proposition~1.1.5.(b) on page~25]{Yuri},
$\nabla f(y)-\nabla f(x)\in N_{C}^{P}(x).$
\end{proof}

The following example illustrates the geodesics
$\{z_{\lambda}:\; 0\leq\lambda \leq 1\}$ given by (\ref{geodesics}).
\begin{example}
Let $x=(x_j)_{1\leq j\leq J}$ and $y=(y_j)_{1\leq j\leq J}$ be
two points in $\RR^J$.
\begin{enumerate}
\item If $f\colon x \mapsto \frac{1}{2}\|x\|^2$, then
$\nabla f=\nabla f^*=\Id.$ We have
$$z_{\lambda}=\lambda y+(1-\lambda)x,$$
for $\lambda\in [0,1]$. Hence $z_{\lambda}$ is a component-wise
\emph{arithmetic mean} of $x$ and $y$.
\item
If $f\colon x\mapsto\sum_{j=1}^{J}x_j\ln(x_j)-x_j$, then
$$\nabla f(x)=(\ln x_{1},\ldots, \ln x_{n}),$$
$$f^*\colon x^*\mapsto \sum_{j=1}^{J}\exp(x_{j}^*),$$ so that
$$\nabla f^*(x^*)=(\exp x_{1}^*,\ldots, \exp x_{J}^*).$$
We have
\begin{align}
z_{\lambda} &=\nabla f^*(\lambda \nabla f(y)+(1-\lambda)\nabla f(x))
\\
&=\nabla f^*(\lambda\ln y_{1}+(1-\lambda)\ln x_{1},\ldots, \lambda \ln y_{J}+(1-\lambda)\ln x_{J})\\
&=(\exp(\lambda\ln y_{1}+(1-\lambda)\ln x_{1}),\ldots, \exp(\lambda \ln y_{J}+(1-\lambda)\ln x_{J}))\\
&=(y_{1}^{\lambda}x_{1}^{1-\lambda},\ldots, y_{J}^{\lambda}x_{J}^{1-\lambda}).
\end{align}
Hence $z_{\lambda}$ is a component-wise \emph{geometric mean} of $x$ and $y$.
\item If $f\colon x\mapsto \sum_{j=1}^{J}\exp(x_{j})$, then
$f^*\colon x^*\mapsto\sum_{j=1}^{J}x_j^*\ln(x_j^*)-x_j^*$ so that
$$\nabla f(x)=(\exp(x_{1}),\ldots, \exp(x_{J})),\quad
\nabla f^*(x^*)=(\ln x_{1}^*,\ldots, \ln x_{J}^*).$$ Hence
$$z_{\lambda}=(\ln (\lambda \exp(y_{1})+(1-\lambda) \exp
(x_{1})),\ldots,\ln (\lambda \exp(y_{J})+(1-\lambda) \exp
(x_{J}))).$$
\end{enumerate}
\end{example}
Define the \emph{symmetrization} of $D$ for $x,y\in U$ by
$$S(x,y):=D(x,y)+D(y,x)=\langle \nabla f(x)-\nabla f(y),x-y\rangle.$$
\begin{proposition}
Given $x,y\in U$ and $0<\lambda < 1$, set
$$z_{\lambda}:=\nabla f^*(\lambda\nabla f(y)+(1-\lambda) \nabla f(x)).$$
Then we have
\begin{enumerate}
\item[{\rm (i)}] $D(x,y)=D(x,z_{\lambda})+D(z_{\lambda},y)+\frac{1-\lambda}{\lambda}S(x,z_{\lambda}).$
\item[{\rm (ii)}]
$S(x,y)=\frac{1}{1-\lambda}S(y,z_{\lambda})+\frac{1}{\lambda}S(z_{\lambda},x).$
\end{enumerate}
\end{proposition}
\begin{proof}
Since $z_{\lambda}=\nabla f^*(\lambda\nabla f(y)+(1-\lambda) \nabla f(x))$, and
$$D(x,z_{\lambda})=f(x)-f(z_{\lambda})-\langle\nabla f(z_{\lambda}),x-z_{\lambda}\rangle,$$
we have
\begin{align}
D(x,z_{\lambda}) &=f(x)-f(z_{\lambda})-\langle \lambda \nabla f(y)+(1-\lambda)\nabla f(x), x-z_{\lambda}
\rangle\\
& =\lambda [f(x)-f(z_{\lambda})-\langle\nabla f(y),x-z_{\lambda}\rangle]+
(1-\lambda)[f(x)-f(z_{\lambda})-\langle\nabla f(x),x-z_{\lambda}\rangle]\\
&= \lambda [D(x,y)-D(z_{\lambda},y)]-(1-\lambda)[f(z_{\lambda})-f(x)-\langle\nabla f(x),z_{\lambda}-x\rangle]
\\
& =\lambda [D(x,y)-D(z_{\lambda},y)]-(1-\lambda) D(z_{\lambda},x).
\end{align}
Hence
$(1-\lambda)[D(x,z_{\lambda})+D(z_{\lambda},x)]+\lambda D(x,z_{\lambda})+\lambda D(z_{\lambda},y)=\lambda D(x,y).$
Dividing both sides by $\lambda$ yields
\begin{equation}\label{oneeq}
D(x,y)=D(x,z_{\lambda})+D(z_{\lambda},y)+\frac{1-\lambda}{\lambda}S(x,z_{\lambda}),
\end{equation}
which is (i).

To see (ii), we rewrite
$$z_{\lambda}=\nabla f^*((1-\lambda) \nabla f(x)+\lambda\nabla f(y)).$$
Applying (i), we get
\begin{equation}\label{twoeq}
D(y,x)=D(y,z_{\lambda})+D(z_{\lambda},x)+\frac{\lambda}{1-\lambda}S(z_{\lambda},y).
\end{equation}
Adding (\ref{oneeq}) and (\ref{twoeq}), we obtain
\begin{align}
S(x,y)& =D(x,y)+D(y,x)\\
&= [D(z_{\lambda},y)+D(y,z_{\lambda})]+[D(z_{\lambda},x)+D(x,z_{\lambda})]+\frac{1-\lambda}{\lambda}S(x,z_{\lambda})\\
   & \qquad +\frac{\lambda}{1-\lambda} S(z_{\lambda},y)\nonumber\\
   &= \bigg(1+\frac{\lambda}{1-\lambda}\bigg)S(z_{\lambda}, y)+\bigg(1+\frac{1-\lambda}{\lambda}\bigg)
   S(x,z_{\lambda})\\
   &=\frac{1}{1-\lambda}S(y,z_{\lambda})+\frac{1}{\lambda}S(z_{\lambda},x),
   \end{align}
   which is (ii).
\end{proof}

\section{Bregman Nearest Points and Maximal Monotone Operators}
\label{monotone}
We shall need the following pointwise version of a concept
due to Rockafellar and Wets \cite[Definition~1.16]{Rock98}.

\begin{definition}
Let
$g:\RR^J\times \RR^J\rightarrow \RX$ and let $\bar{y}\in\RR^J$. We say
that $g$ is level bounded in the first variable locally uniformly at
$\bar{y}$, if for every $\alpha\in\RR$, there exists $\delta>0$ such that
$$
\bigcup_{y\in B_\delta(\bar{y})}
\{x\in \RR^J\colon g(x,y)\leq \alpha\}\;\;\text{is bounded.}$$
\end{definition}

\begin{proposition}\label{generalcase}
The Bregman distance $D$ is level bounded in the first variable locally
uniformly at every point in $U$.
\end{proposition}

\begin{proof}
Suppose the opposite. Then, for some $\bar{y}\in U$, $\bar{\alpha}\in\RR$, for
every $n\in \{1,2,\ldots\}$, there exist $y_{n}\in U$, $x_{n}\in\dom
f$ such that
$$\|y_{n}-\bar{y}\|<\frac{1}{n},\quad  D(x_{n},y_{n})\leq \bar{\alpha},\quad  \|x_{n}\|>n.$$
We then have $y_{n}\rightarrow\bar{y},$ $\|x_{n}\|\rightarrow\infty$
and
\begin{equation}\label{finitevalue}
 \quad D(x_{n},y_{n})\leq \bar{\alpha}.
\end{equation}
Now
\begin{align}
D(x_{n},y_{n})& =f(x_{n})-f(y_{n})-\scal{\nabla f(y_{n})}{x_{n}-y_{n}}\\
&= f(x_{n})-\scal{\nabla f(y_{n})}{x_{n}}+[-f(y_{n})+\scal{\nabla f(y_{n})}{y_{n}}].\label{timhorton}
\end{align}
Since $f$ is Legendre, $\nabla f$ is continuous on $U$. When $n\rightarrow\infty$, we have
\begin{equation}\label{finitevalue1}
-f(y_{n})+\scal{\nabla f(y_{n})}{y_{n}}\rightarrow -f(\bar{y}) +\scal{\nabla f(\bar{y})}{\bar{y}},
\end{equation}
and
\begin{align}
f(x_{n})-\scal{\nabla f(y_{n})}{x_{n}} &=\|x_{n}\|\left(\frac{f(x_{n})}{\|x_{n}\|}-\langle\nabla f(y_{n}),
\frac{x_{n}}{\|x_{n}\|}\rangle\right)\label{finitevalue2}\\
&\geq \|x_{n}\|\left(\frac{f(x_{n})}{\|x_{n}\|}-\|\nabla f(y_{n})\|\right)\rightarrow\infty,\label{finitevalue3}
\end{align}
since $\|\nabla f(y_{n})\|\rightarrow \|\nabla f(\bar{y})\|$ and
$\lim f(x_{n})/\|x_{n}\|=+\infty.$
(\ref{finitevalue1})-(\ref{finitevalue3}) and (\ref{timhorton}) altogether show that $D(x_{n},y_{n})\rightarrow\infty$, but this contradicts (\ref{finitevalue}).
\end{proof}

The following result will be very useful later.

\begin{theorem}\label{banffday} The following hold.
\begin{enumerate}
\item For each $y\in U$, the set
$\bproj{C}(y)$ is nonempty and compact.
Moreover, $\bD{C}$ is continuous on $U$.
\item If $x_{n}\in \bproj{C}(y_{n})$ and $y_{n}\rightarrow y\in U$, then
the sequence $(x_{n})_{n=1}^{\infty}$ is bounded, and all its cluster points lie in $\bproj{C}(y)$.
\item Let $y\in U$ and $\bproj{C}(y)=\{x\}$. If $x_{n}\in \bproj{C}(y_{n})$ and
$y_{n}\rightarrow y$, then $x_{n}\rightarrow x$; consequently, $\bproj{C}$
is continuous at $y$.
\end{enumerate}
\end{theorem}
\begin{proof}  Fix $\bar{y}\in U$ and $\delta>0$ such that
$B_\delta(\bar{y})\subset U$. Consider the proper lower semicontinuous function $g:\RR^J\times \RR^J\rightarrow \RX$ defined by
$$(x,y)\mapsto D(x,y)+\iota_{C}(x)+\iota_{B_\delta(\bar{y})}(y).$$
Observe that $\dom g=C\times B_\delta(\bar{y})$.
For every $y \in \RR^J$ and $\alpha\in\RR$, we have
\begin{equation} \label{e:071209:a}
\{x\in\RR^J \colon g(x,y)\leq\alpha\} =
\begin{cases}
C\cap \{x\in\RR^J\colon D(x,y)\leq\alpha\}, &\text{if $y\in
B_\delta(\bar{y})$;}\\
\varnothing, &\text{otherwise.}
\end{cases}
\end{equation}
We now show that
\begin{equation} \label{e:071209:b}
\text{$g$ is level bounded in the first variable
locally uniformly at every point in $\RR^J$. }
\end{equation}
To this end, fix $\bar{z}\in\RR^J$ and $\alpha\in\RR$.\\
\emph{~~Case 1:} $\bar{z}\notin B_\delta(\bar{y})$.\\
Let $\epsilon>0$ be so small
that $B_\delta(\bar{y})\cap B_\epsilon(\bar{z})=\varnothing$.
Then \eqref{e:071209:a} yields
$$
\bigcup_{z\in B_\epsilon(\bar{z})} \{x\in\RR^J \colon g(x,z)\leq\alpha\} =
\varnothing,
$$
which is certainly bounded.\\
\emph{~~Case 2:} $\bar{z}\in B_\delta(\bar{y})$.\\
Since $B_\delta(\bar{y})\subset U$, we have $\bar{z}\in U$.
Proposition~\ref{generalcase} guarantees the existence of
$\epsilon>0$ such that
$$\bigcup_{z\in B_\epsilon(\bar{z})} \{x\in\RR^J \colon D(x,z)\leq
\alpha\}\;\;\text{is bounded.}$$
In view of \eqref{e:071209:a}, the set
$$\bigcup_{z\in B_\epsilon(\bar{z})\cap B_\delta(\bar{y})}
C \cap \{x\in\RR^J\colon D(x,z)\leq \alpha\} = \bigcup_{z\in
B_\epsilon(\bar{z})}\{x\in\RR^J \colon g(x,z)\leq \alpha\}
$$
is bounded as well.

Altogether, we have verified \eqref{e:071209:b}.

Define a function $m$ at
$y\in\RR^J$ by
$$m(y):=\inf_{x\in\RR^J} g(x,y)=\begin{cases}
\inf_{x\in C}D(x,y)=\bD{C}(y), &
\text{if $y\in B_\delta(\bar{y})$;}\\
{+\infty}, &\text{otherwise.}
\end{cases}$$
Then $m = \bD{C} + \iota_{B_\delta(\bar{y})}$ and
 $$\argmin_{x\in\RR^J}g(x,y)=\begin{cases}
\argmin_{x\in C}D(x,y)=\bproj{C}(y), &
\text{if $y\in B_\delta(\bar{y})$;}\\
\varnothing, &\text{otherwise.}
\end{cases}
$$
Now \eqref{e:071209:b} and
\cite[Theorem~1.17(a)]{Rock98} implies that
if $y\in B_\delta(\bar{y})$, then $\bproj{C}(y)$ is nonempty and compact.
In particular, $\bproj{C}(\bar{y})\neq\varnothing$ and compact.
Take $\bar{x}\in\bproj{C}(\bar{y})$. As
$$g(\bar{x},\cdot)=D(\bar{x},\cdot)+\iota_{B_\delta(\bar{y})}$$
is continuous at $\bar{y}$, by \cite[Theorem~1.17(c)]{Rock98},
the function $m$ is continuous at $\bar{y}$.
Hence $\bD{C}$ is continuous at $\bar{y}$.
Since $\bar{y}\in U$ is arbitrary, this proves (i). Next,
\cite[Theorem~1.17(b)]{Rock98} gives (ii) since
$\bD{C}$ is continuous on $U$.

Finally, (iii) is an immediate consequence of (ii).
\end{proof}

Our next result states that $\bproj{C}\circ\nabla f^*$ is a monotone operator.
This is also related to \cite[Proposition~3.32.(ii)(c)]{Sico03},
which establish a stronger property when $C$ is convex.

\begin{proposition}\label{monotoneyes}
For every $x,y$ in $U$,
\begin{equation}\label{seehow}
 \langle \bproj{C}(y)-\bproj{C}(x),\nabla f(y)-\nabla f(x)\rangle\geq 0;
\end{equation}
consequently, $\bproj{C}\circ \nabla f^*$ is monotone.
\end{proposition}
\begin{proof} Since
$$D(\bproj{C}(x),y)\geq D(\bproj{C}(y),y),\quad D(\bproj{C}(y),x)\geq D(\bproj{C}(x),x),$$
we use (\ref{3point}) to get
$$f(\bproj{C}(x))-f(\bproj{C}(y))-\langle\nabla f(y), \bproj{C}(x)-\bproj{C}(y)\rangle\geq 0,$$
$$f(\bproj{C}(y))-f(\bproj{C}(x))-\langle\nabla f(x), \bproj{C}(y)-\bproj{C}(x)\rangle\geq 0.$$
Adding these inequalities yields
$$\langle\nabla f(y), \bproj{C}(y)-\bproj{C}(x)\rangle -\langle\nabla f(x), \bproj{C}(y)-\bproj{C}(x)\geq 0,
$$
i.e., \eqref{seehow}. The monotonicity now follows from Fact~\ref{isom}
and our assumption that $\dom f^*=\RR^J$.
\end{proof}

\begin{definition}
The set $C$ is \emph{Chebyshev with respect to the left Bregman distance},
or simply \bDC,
if for every $x\in U$,
$\bproj{C}(x)$ is nonempty and a singleton.
\end{definition}
For some instances of $f$, it is known that if $C$ is convex,
then it is \bDC\ (see, e.g., \cite[Theorem~3.14]{Baus97})
and $\bproj{C}$ is continuous (see, e.g., \cite[Proposition~3.10(i)]{noll}).
The next result is a refinement.

\begin{proposition}\label{maximalcase}
Suppose that $C$ is \bDC. Then
$\bproj{C}: U\rightarrow C$ is continuous.
Hence $\bproj{C}\circ \nabla f^*$ is continuous and
maximal monotone.
\end{proposition}
\begin{proof}
While the continuity of $\bproj{C}$ follows from Theorem~\ref{banffday}(iii), Proposition~\ref{monotoneyes} shows that $\bproj{C}\circ \nabla f^*$ is
monotone. Since $\bproj{C}$ is continuous on $U$ and
$\nabla f^*:\RR^J\rightarrow
U $ is continuous, we conclude that $\bproj{C}\circ\nabla f^*$ is continuous on $\RR^J$.
Altogether, since $\bproj{C}\circ \nabla f^*$ is single-valued, it is
maximal monotone on $\RR^J$ by \cite[Example~12.7]{Rock98}.
\end{proof}

Rockafellar's well-known result on the virtual convexity of the
range of a maximal monotone operator allows us to show that \bDC\ sets
are convex. Our proof extends a Hilbert space technique due to Berens and
Westphal \cite{BW}.

\begin{theorem}[\bDC\ sets are convex]
Suppose that $C$ is \bDC. Then $C$ is convex.
\end{theorem}
\begin{proof}
By Proposition~\ref{maximalcase}, $\bproj{C}\circ \nabla f^*$ is a
maximal monotone operator on $\RR^J$. Using
\cite[Theorem 12.41]{Rock98} (or \cite[Theorem~19.2]{Simons}), $\closu [\ran
\bproj{C}\circ \nabla f^*]$ is convex. Since $ \ran \nabla f^*=U$
and $C\subset U $,  it follows that
$$C\supset\ran\big(\bproj{C}\circ \nabla f^*\big)=\bproj{C}(\nabla f^*(\RR^J))= \bproj{C}(U)\supset \bproj{C}(C)=C,$$
from which $\closu [\ran \bproj{C}\circ \nabla f^*]=\closu C=C.$
Hence $C$ is convex.
\end{proof}

\begin{corollary}
The set $C$ is \bDC\ if and only if it is convex.
\end{corollary}

\section{Subdifferentiabilities of Bregman Distances}\label{clarkemor}

Let us show that $\bD{C}$  is locally Lipschitz on $U$.

\begin{proposition}\label{distance}
Suppose $f$ is twice continuously differentiable on $U$.
Then the left Bregman distance function satisfies
\begin{equation}
\bD{C} =f^*\circ\nabla f-(f+\iota_{C})^*\circ \nabla f
=[f^*-(f+\iota_{C})^*]\circ \nabla f,
\end{equation}
and it is locally Lipschitz on $U$.
\end{proposition}
\begin{proof}
The Mean Value Theorem and the continuity of $\Hess f$ on $U$ imply
that $\nabla f$ is locally Lipschitz on $U$.
For $y\in U$,
\begin{align}
\bD{C}(y) &=\inf_{c\in C}[f(c)-f(y)-\langle\nabla f(y),c-y\rangle]\\
&=\inf_{c}[(f+\iota_{C})(c)-\langle\nabla f(y),c\rangle +f^*(\nabla f(y))]\\
&=f^*(\nabla f(y))-\sup_{c}[\langle\nabla f(y),c\rangle-(f+\iota_{C})(c)]\\
& =f^*(\nabla f (y))-(f+\iota_{C})^*(\nabla f(y)).
\end{align}
Note that $f+\iota_{C}\geq f$, $(f+\iota_{C})^*\leq f^*$, so $\dom
f^*\subset \dom (f+\iota_{C})^*$. Being convex functions, both
$(f+\iota_{C})^*$ and $f^*$ are locally Lipschitz on interior of
their respective domains, in particular on $\inte\dom f^*=\RR^J$. Since
$\nabla f: U \rightarrow\RR^J$ is locally Lipschitz, we
conclude that $\bD{C}$ is locally Lipschitz on $U$.
\end{proof}

For a function $g$ that is finite and locally Lipschitz at a point $y$,
we define
the \emph{Dini subderivative} and \emph{Clarke subderivative} of $g$ at $y$ in the direction $w$, denoted respectively
by $\deriv g(y)(w)$ and $\rockderiv g (y)(w)$, via
$$\deriv g(y)(w):=\liminf_{t\downarrow 0}\frac{g(y+tw)-g(y)}{t},$$
$$\rockderiv g(y)(w):=\limsup_{\stackrel{x\rightarrow y}{t\downarrow 0}}\frac{g(x+tw)-g(x)}{t},$$
and the corresponding \emph{Dini subdifferential} and
\emph{Clarke subdifferential} via
$$\hat{\partial} g(y): =\{y^*\in\RR^J: \langle y^*,w\rangle\leq \deriv g(y)(w),\; \forall w\in\RR^J\},$$
$$\overline{\partial} g(y): =\{y^*\in\RR^J: \langle y^*,w\rangle\leq \rockderiv g(y)(w),\; \forall w\in\RR^J\}.$$
Furthermore, the \emph{limiting subdifferential} is
defined by
$$\partial_L g(y):=\limsup_{x\rightarrow y}\hat{\partial}
g(x),$$ see \cite[Definition~8.3]{Rock98}. We say that $g$ is \emph{Clarke
regular} at $y$ if $\deriv g(y)(w)=\rockderiv g(y)(w)$ for every
$w\in \RR^J$, or equivalently $\hat{\partial}
g(y)=\overline{\partial} g(y)$.
For further properties
of these subdifferentials and subderivatives,
see \cite{Frank,mordukhovich,Rock98}.

We now study the
subdifferentiability of $\bD{C}$ in terms of $\bproj{C}$.
\begin{proposition}\label{subdiff1}
Suppose $f$ is twice continuously differentiable on $U$.
Then the function $-\bD{C}$ is Dini subdifferentiable on
$U$; more precisely, if $y\in U$, then
$$\Hess f(y)[\bproj{C}(y)-y]\subset \hat{\partial} (-\bD{C})(y),$$
and thus
\begin{equation}\label{tom1}
\Hess f(y)[\conv\bproj{C}(y)-y]\subset \hat{\partial} (-\bD{C})(y).
\end{equation}
\end{proposition}
\begin{proof} Fix $y\in U$.
By Theorem~\ref{banffday}(i), $\bproj{C}(y)\neq \varnothing$. Let $x\in
\bproj{C}(y).$ As $\hat{\partial}$ is convex-valued, it suffices to show that
\begin{equation}\label{subnegative}
\Hess f(y)(x-y)\in \hat{\partial} (-\bD{C})(y).
\end{equation}
To this end, let $t>0$ and $w\in \RR^J$. Since for sufficiently small $t$, $y+tw\in U$,
\begin{align}-\bD{C}(y+tw) &=\sup_{c\in C}\big(-f(c)+f(y+tw)+\langle\nabla
f(y+tw),c-(y+tw)\rangle\big)\\
&\geq -f(x)+f(y+tw)+\langle\nabla f(y+tw),x-(y+tw)\rangle
\end{align}
and
\begin{equation}
\bD{C}(y)=f(x)-f(y)-\langle\nabla f(y),x-y\rangle,
\end{equation}
we have
$$
-\bD{C}(y+tw)+\bD{C}(y)\geq f(y+tw)-f(y)+\langle \nabla f(y+tw)-\nabla f(y),x-y\rangle
+\langle \nabla f(y+tw),-tw\rangle.
$$
Dividing both sides by $t$ and taking the limit inferior
as $t\downarrow 0$, we have
\begin{align}
\deriv (-\bD{C})(y)(w) &\geq \langle\nabla f(y),w\rangle +\langle \Hess f(y) w, x-y\rangle
-\langle \nabla f(y),w\rangle\\
&=\langle \Hess f(y)(x-y), w\rangle,
\end{align}
which gives (\ref{subnegative}).
\end{proof}

\begin{lemma}\label{subdiff2}
Suppose that $f$ is twice continuously differentiable on $U$,
let $y\in U$, and suppose that $\bproj{C}(y)$ is a singleton.
Then $\bD{C}$ is Dini subdifferentiable at $y$ and
\begin{equation}\label{amy}
\Hess f(y)(y-\bproj{C}(y))\in \hat{\partial} \bD{C}(y).
\end{equation}
\end{lemma}
\begin{proof}
Suppose that $\bproj{C}(y)=\{x\}$, and fix $w\in \RR^{J}$.
Let $(t_{n})$ be a positive sequence such that
$(y+t_nw)$ lies in $U$, $t_{n}\downarrow 0$, and
$$\deriv \bD{C}(y)(w)=\lim_{n\rightarrow \infty}\frac{\bD{C}(y+t_{n}w)-\bD{C}(y)}{t_{n}}.$$
Select $x_{n}\in \bproj{C}(y+t_{n}w)$, which is possible by Theorem~\ref{banffday}(i).
We have
\begin{multline}
\bD{C}(y+t_{n}w)-\bD{C}(y)\nonumber\\
\begin{aligned}
 =&\; D(x_{n},y+t_{n}w)-D(x_{n},y)+D(x_{n},y)-D(x,y)\\
\geq &\; D(x_{n},y+t_{n}w)-D(x_{n},y)\\
= &\; f(x_{n})-f(y+t_{n}w)-\langle\nabla f(y+t_{n}w),x_{n}-(y+t_{n}w)\rangle
-[f(x_{n})-f(y)-\langle \nabla f(y),x_{n}-y\rangle] \\
=&\; -(f(y+t_{n}w)-f(y))-\langle \nabla f(y+t_{n}w)-\nabla f(y),x_{n}-y\rangle +t_{n}
\langle\nabla f(y+t_{n}w),w\rangle .\\
\end{aligned}
\end{multline}
Dividing by $t_{n}$, we get
\begin{multline}\label{weneed}
\frac{\bD{C}(y+t_{n}w)-\bD{C}(y)}{t_{n}}\geq \\
-\frac{f(y+t_{n}w)-f(y)}{t_{n}}
-\frac{\langle \nabla f(y+t_{n}w)-\nabla f(y),x_{n}-y\rangle}{t_{n}}+\langle\nabla f(y+t_{n}w),w\rangle.
\end{multline}
By Theorem~\ref{banffday}(iii), $x_{n}\rightarrow x$.
Taking limits in (\ref{weneed}) yields
$$\deriv \bD{C}(y)(w)\geq -\langle \Hess f(y)w, x-y\rangle=\langle \Hess f(y)(y-x),w\rangle.$$
Since this holds for every $w\in \RR^J$, we conclude that $\Hess f(y)(y-x)\in \hat{\partial}\bD{C}(y)$.
\end{proof}

Lemma~\ref{subdiff2} allows us to generalize
\cite[Example~8.53]{Rock98} from the Euclidean distance to
the left Bregman distance. It delineates the differences between 
the Dini subdifferential, limiting subdifferential and Clarke subdifferential.

\begin{theorem}\label{complete}
Suppose that $f$ is twice continuously differentiable on $U$ and that
for every $u\in U$, $\Hess f(u)$ is positive definite.
Set $g=\bD{C}$, and let $y\in U$ and $w\in\RR^J$.
Then the following hold.
\begin{enumerate}
\item The Dini subderivative is
\begin{equation}\label{rightone}
\deriv g(y)(w)=\min_{x\in\bproj{C}(y)}\langle \Hess f(y)(y-x),w\rangle,
\end{equation}
so that the Dini subdifferential of $g$ is
\begin{equation}\label{theset}
\hat{\partial} g(y)=\begin{cases}
\{\Hess f(y)[y-\bproj{C}(y)]\} & \mbox{ if $\bproj{C}(y)$ is a singleton;}\\
\varnothing, & \mbox{ otherwise}.
\end{cases}
\end{equation}
The limiting subdifferential is
\begin{equation}\label{mordusubdiff}\partial_L g(y)=\Hess f(y)[y-\bproj{C}(y)].
\end{equation}
The Clarke subderivative is
\begin{equation}\label{clarkederiv}
\rockderiv g(y) (w)=\max_{x\in\bproj{C}(y)}\langle \Hess f(y)(y-x), w\rangle,
\end{equation}
from which we get the Clarke subdifferential
\begin{equation}\label{tom2}\overline{\partial}g(y)=\Hess f(y)[y-\conv\bproj{C}(y)].
\end{equation}
Hence  $-\bD{C}$ is Clarke regular on $U$.
\item If $y\in C$, then
$g$ is strictly differentiable with derivative $0$.
\end{enumerate}
\end{theorem}
\begin{proof} By Theorem~\ref{banffday}(i), $\bproj{C}(y)\neq\varnothing$.
Fix $x\in \bproj{C}(y)$ and $t>0$ sufficiently small so that $y+tw\in U$.
In view of $\bD{C}(y+tw)\leq D(x,y+tw)$ and
$\bD{C}(y)=D(x,y)$, we have
\begin{multline}
\deriv g(y)(w)= \liminf_{t\downarrow 0}\frac{\bD{C}(y+tw)-\bD{C}(y)}{t}\leq \liminf_{t\downarrow 0}\frac{D(x,y+tw)-D(x,y)}{t}\nonumber\\
\begin{aligned}
= &\;\liminf_{t\downarrow 0}\frac{ f(x)-f(y+tw)-\langle \nabla f(y+tw),x-(y+tw)\rangle-[f(x)-f(y)-\langle\nabla f(y), x-y\rangle]}{t} \\
= &\; \liminf_{t\downarrow 0}\frac{-[f(y+tw)-f(y)]-\langle \nabla f(y+tw)-\nabla f(y), x-y\rangle+t\langle\nabla f(y+tw), w\rangle}{t}\\
= &\; \liminf_{t\downarrow 0}-\frac{f(y+tw)-f(y)}{t}-\frac{\langle \nabla f(y+tw)-\nabla f(y), x-y\rangle}{t}+\langle\nabla f(y+tw), w\rangle\\
=&\; -\langle\nabla f(y), w\rangle-\langle \Hess f(y)w,x-y\rangle+\langle\nabla f(y),w\rangle\\
= &\; \langle \Hess f(y)(y-x),w\rangle.
\end{aligned}
\end{multline}
Since this holds for every $x\in\bproj{C}(y)$, it follows from Theorem~\ref{banffday}(i) that
$$\deriv g (y)(w)\leq \min_{x\in\bproj{C}(y)}\langle \Hess f(y)(y-x),w\rangle.$$

To get the opposite inequality, we consider a positive sequence $(t_n)$
such that $t_{n}\downarrow 0$, $(y+t_nw)$ lies in $U$, and
$$\deriv g(y)(w)=\lim_{n\rightarrow\infty} \frac{\bD{C}(y+t_{n}w)-\bD{C}(y)}{t_{n}}.$$
Select $x_{n}\in\bproj{C}(y+t_{n}w)$, which is possible by
Theorem~\ref{banffday}(i). Then
\begin{align}
\bD{C}(y+t_{n}w) &=D(x_{n},y+t_{n}w)\nonumber\\
& =f(x_{n})-f(y+t_{n}w)-\langle\nabla f(y+t_{n}w),x_{n}-(y+t_{n}w)\rangle
\end{align}
and
\begin{equation}
\bD{C}(y)\leq D(x_{n},y) \leq f(x_{n})-f(y)-\langle \nabla f(y),x_{n}-y\rangle.
\end{equation}
By Theorem~\ref{banffday}(ii), and after taking a subsequence if necessary,
we assume that $x_{n}\rightarrow x\in \bproj{C}(y)$. We estimate
\begin{multline}
\frac{\bD{C}(y+t_{n}w)-\bD{C}(y)}{t_{n}}\\
\begin{aligned}
\geq &\; \frac{-[f(y+t_{n}w)-f(y)]-\langle \nabla f(y+t_{n}w)-\nabla f(y),x_{n}-y\rangle+\langle\nabla
f(y+t_{n}w),t_{n}w\rangle}{t_{n}}\\
=& \frac{-[f(y+t_{n}w)-f(y)]}{t_{n}}-\frac{\langle \nabla f(y+t_{n}w)-\nabla f(y),x_{n}-y\rangle}{t_{n}}+\langle\nabla
f(y+t_{n}w),w\rangle.
\end{aligned}
\end{multline}
Taking limits, we obtain
$$\deriv g(y)(w)\geq -\langle \Hess f(y)w,x-y\rangle =\langle \Hess f(y)(y-x),w\rangle\geq
\min_{x\in\bproj{C}(y)}\langle \Hess f(y)(y-x),w\rangle.$$
Therefore, (\ref{rightone}) is correct.

For $y^*\in\RR^J$, $y^*\in \hat{\partial}g(y)$ if and only if
$$\langle y^*,w\rangle\leq \langle \Hess f(y)(y-x),w\rangle \quad \forall
x\in\bproj{C}(y),w\in \RR^J.$$
This holds if and only if $y^*=\Hess f(y)(y-x)$, $\forall x\in\bproj{C}(y)$;
since $\Hess f(y)$ is invertible, we deduce that $x=y-(\Hess f(y))^{-1}y^*$,
so that $\bproj{C}(y)$ is unique. Therefore, if $\bproj{C}(y)$ is not unique,
then $\hat{\partial} g(y)$ has to be empty. Hence (\ref{theset}) holds.

For every $z\in \RR^J$, we have
$$\hat{\partial} g(z)\subset \Hess f(z)(z-\bproj{C}(z)).$$
The upper semicontinuity of $\bproj{C}$ (see Theorem~\ref{banffday}(ii))
implies through
$\partial_L g(y)=\limsup_{z\rightarrow y}\hat{\partial} g(z)$
that
\begin{equation}\label{hockey1}\partial_L g(y)\subset \Hess f(y)(y-\bproj{C}(y)).
\end{equation}
Equality actually has to hold. Indeed, for
$x\in\bproj{C}(y)$ and $0\leq \lambda< 1$,
the point
$$z_{\lambda}:=\nabla f^*(\lambda\nabla f(y)+(1-\lambda)\nabla f(x)),$$
has $\bproj{C}(z_{\lambda})=\{x\}$ by Proposition~\ref{near}(ii).
Lemma~\ref{subdiff2} shows that
$$\Hess f(z_{\lambda})(z_{\lambda}-x)\in \hat{\partial} g(z_{\lambda}),$$
where $\Hess f(z_\lambda)(z_\lambda-x)\to \Hess f(y)(y-x)$
as $\lambda \to 1$, since $\Hess f$ is continuous. Thus
$\Hess f(y)(y-x)\in \partial_L g(y)$ and therefore
\begin{equation}\label{hockey2}
\Hess f(y)(y-\bproj{C}(y))\subset \partial_L g(y).
\end{equation}
Hence (\ref{hockey1}) and (\ref{hockey2}) together
give (\ref{mordusubdiff}).

Since $g$ is locally Lipschitz around $y\in U$
by Proposition~\ref{distance}, the singular subdifferential of $g$ at $y$ is $0$, so that its polar cone
is $\RR^J$. Then for every $w\in \RR^J$, using  \cite[Exercise~8.23]{Rock98} we have
$$\rockderiv g(y)(w)=\sup\{\langle y^*,w\rangle:\; y^*\in\partial_L g(y)\};$$
thus, (\ref{clarkederiv}) follows from (\ref{mordusubdiff}). Now (\ref{clarkederiv}) is the same as
$$\rockderiv g(y) (w)=\max\langle \Hess f(y)(y-\conv\bproj{C}(y)), w\rangle.$$
As $\conv\bproj{C}(y)$ is compact, we obtain (\ref{tom2}). Or directly apply \cite[Theorem 8.49]{Rock98}
and (\ref{mordusubdiff}).

The Clarke regularity of $-\bD{C}$ follows from combining (\ref{tom1})
and (\ref{tom2}). Indeed,
$$\Hess f(y)[\conv \bproj{C}(y)-y]\subset
\hat{\partial}(-\bD{C})(y)\subset\overline{\partial}
(-\bD{C})(y)=\Hess f(y)[\conv\bproj{C}(y)-y],$$ so that $\hat{\partial}
(-\bD{C})(y)=\overline{\partial}(-\bD{C})(y).$

(ii): When $y\in C$, $\bproj{C}(y)=\{y\}$. By (\ref{mordusubdiff}),
$\partial_L g(y)=\{0\}$, and this implies that $g$ is strictly
differentiable at $y$ by \cite[Theorem~9.18(b)]{Rock98}.
\end{proof}

\begin{corollary}\label{singletonchar}
Suppose that $f$ is twice continuously differentiable and that
$\Hess f(y)$ is positive definite, for every $y\in U$.
Then for $y\in U$, the following are equivalent:
\begin{enumerate}
\item $\bD{C}$ is Dini subdifferentiable at $y$; %old (i)
\item $\bD{C}$ is differentiable at $y$; %old (ii)
\item $\bD{C}$ is strictly differentiable at $y$;
\item $\bD{C}$ is Clarke regular at $y$;
\item $\bproj{C}(y)$ is a singleton. % old (iii)
\end{enumerate}
If these hold, we have $\nabla \bD{C}(y)=\Hess f(y)[y-\bproj{C}(y)]$.
\end{corollary}
\begin{proof}
(i)$\Rightarrow$(ii): By Proposition~\ref{subdiff1},
both $-\bD{C}$ and $\bD{C}$ are Dini subdifferentiable.
Thus $\bD{C}$ is differentiable at $y$ (see \cite[Exercise~3.4.14 on
page~143]{Yuri}), and
$$\hat{\partial}\bD{C}(y)=-\hat{\partial} (-\bD{C})(y)=\{\nabla \bD{C}(y)\}.$$
(ii)$\Rightarrow$(i) is clear.
(ii)$\Leftrightarrow$(iii)$\Leftrightarrow$(iv): 
This is a consequence of \cite[Theorem~3.4]{WuYe}.
(ii)$\Leftrightarrow$(v):
If $\bD{C}$ is differentiable at $y$, then (\ref{theset}) implies that
$\bproj{C}(y)$ is a singleton. Conversely, if $\bproj{C}(y)$ is a singleton,
then (\ref{tom2}), Proposition~\ref{distance}, and
\cite[Proposition~2.2.4]{Clarke} show that
$\bD{C}$ is strictly differentiable
and hence differentiable at $y$.
Finally, the gradient formula
$\nabla \bD{C}(y)=\Hess f(y)[y-\bproj{C}(y)]$ is a consequence of Proposition~\ref{subdiff1} or Lemma~\ref{subdiff2}.
\end{proof}

\begin{corollary}
Suppose that $f$ is twice continuously differentiable on $U$ and that
$\Hess f(y)$ is positive definite for every
$y\in U$. Then $\bproj{C}$ is almost everywhere and generically single-valued on $U$.
\end{corollary}
\begin{proof}
By Proposition~\ref{distance}, $\bD{C}$ is locally Lipschitz on
$U$. Apply Rademacher's Theorem
\cite[Theorem~9.1.2]{lewis} or
\cite[Corollary~3.4.19]{Yuri}
to obtain that $\bD{C}$ is
differentiable almost everywhere on $U$.
Moreover, since $-\bD{C}$ is Clarke regular
on $U$ by Theorem~\ref{complete}, 
we use \cite[Theorem~10]{loewen} to conclude that
$-\bD{C}$ is differentiable generically on $U$, and so is $\bD{C}$.
Hence the result follows
from Corollary~\ref{singletonchar}.
\end{proof}

\section{Characterizations of Chebyshev Sets}
\label{completecheby}

\begin{definition}\label{fenchelsubdifferential}
For $g:\RR^J\rightarrow\RX$ (not necessarily convex), we let
$$\partial g(x):=\{s\in \RR^J:\; g(y)\geq g(x)+\langle s,y-x\rangle\ \forall y\in\RR^J\}\quad \mbox{if $x\in \dom g$}; $$
and $\partial g(x) = \varnothing$ otherwise;
and the Fenchel conjugate of $g$ is defined by
$$s\mapsto g^*(s):=\sup\{\langle s,x\rangle -g(x):\; x\in \RR^J\}.$$
\end{definition}
According to \cite[Proposition~1.4.3]{urruty1},
\begin{equation}\label{onewayonly1}
s\in\partial g(x)\quad \Rightarrow \quad x\in\partial g^*(s),
\end{equation}
which becomes ``$\Leftrightarrow$'' if $g\in\Gamma$.
In order to study \bDC\ sets, we need two preparatory results concerning subdifferentiabilities
of $f+\iota_{C}$ and $(f+\iota_{C})^*$. Lemmas~\ref{friday}, ~\ref{conjugatesub} and Theorem~\ref{mainpart}
below generalize respectively, and are inspired by
\cite[Propositions~3.2.1, 3.2.2 and Theorem~3.2.3]{urruty1}.

\begin{lemma}\label{friday}
Let $x\in \RR^J$. Then
$$ \partial (f+\iota_{C})(x)=\{s\in\RR^J\colon x\in\bproj{C}(\nabla f^*(s))\}=(\bproj{C}\circ\nabla f^*)^{-1}(x),$$
and consequently $\partial (f+\iota_C) = \big(\bproj{C}\circ\nabla
f^*\big)^{-1}$.
\end{lemma}
\begin{proof} The statement is clear if $x\notin C$, so assume $x\in C$.
By \cite[Theorem~1.4.1]{urruty1},
\begin{equation}\label{gonewithwind}
s\in \partial (f+\iota_{C})(x) \quad \Leftrightarrow \quad
(f+\iota_{C})^*(s)+(f+\iota_{C})(x)=\langle s,x\rangle.
\end{equation}
Proposition~\ref{distance} shows that
$$(f+\iota_{C})^*=f^*-\bD{C}\circ\nabla f^* \quad \mbox{ on $\RR^J$}.$$
Combining with (\ref{gonewithwind}) and since $x\in C$, we get
$$s\in \partial (f+\iota_{C})(x)\quad \Leftrightarrow\quad f^*(s)-(\bD{C}\circ\nabla f^*)(s)+f(x)=\langle s,x\rangle;$$
equivalently,
\begin{align}
\bD{C}(\nabla f^*(s))& =f(x)+f^*(s)-\langle s,x\rangle\\
& =f(x)+f^*((\nabla f\circ \nabla f^*)(s))-\langle\nabla f\circ \nabla f^*(s),x\rangle\\
&= f(x)-f(\nabla f^*(s))-\langle\nabla f(\nabla f^*(s)), x-\nabla f^*(s)\rangle\\
&= D(x,\nabla f^*(s)),
\end{align}
i.e., $x\in\bproj{C}(\nabla f^*(s)).$
\end{proof}

The following result,
which establishes the link between $\partial (f+\iota_{C})^*$ and
$\bproj{C}\circ\nabla f^*$,
is the cornerstone for the convexity characterization of \bDC\ sets.
\begin{lemma} \label{conjugatesub}
Let $s\in\RR^J$. Then
$$\partial (f+\iota_{C})^*(s)=\conv [\bproj{C}(\nabla f^*(s))]=\conv [\bproj{C}\circ\nabla f^*(s)]
.$$
Consequently, $\bproj{C}\circ\nabla f^*$ is monotone on $\RR^{J}$.
\end{lemma}
\begin{proof}
Since $f$ is $1$-coercive and $C$ is closed, the function
$f+\iota_{C}$ is $1$-coercive and lower semicontinuous.
We have that $\conv (f+\iota_{C})$ is lower semicontinuous
by \cite[Proposition~1.5.4]{urruty1}, and
$\dom (f+\iota_{C})^{*}=\RR^J$ by \cite[Proposition~1.3.8]{urruty1}.
Now
$$x\in \partial (f+\iota_{C})^*(s) \quad \Leftrightarrow\quad x\in \partial [\conv (f+\iota_{C})]^*(s)
\quad \Leftrightarrow\quad s\in \partial [\conv(f+\iota_{C})](x),$$
in which the first equivalences follows from \cite[Corollary~1.3.6]{urruty1} and the second equivalence uses the lower semicontinuity of $\conv (f+\iota_{C})$.
Using \cite[Theorem~1.5.6]{urruty1},
$s\in \partial [\conv(f+\iota_{C})](x)$ if and only if there exist $x_{1},\ldots, x_{k}\in \RR^J$,
$\alpha_{1},\ldots,\alpha_{k}>0$ such that
\begin{equation}\label{togo}
\sum_{j=1}^{k}\alpha_j=1, \quad
x=\sum_{j=1}^{k}\alpha_{j}x_{j},\quad \mbox{ and }\quad s\in \bigcap_{j=1}^{k}\partial (f+\iota_{C})(x_{j}).
\end{equation}
But $s\in \partial (f+\iota_{C})(x_{j})$ is equivalent to $$x_{j}\in \bproj{C}(\nabla f^*(s)),$$
by Lemma~\ref{friday}.
Hence (\ref{togo}) gives $\partial (f+\iota_{C})^*(s)=\conv \bproj{C}(\nabla f^*(s)).$
Finally, as a selection of $\partial(f+\iota_C)^*$, which is maximal
monotone, the operator $\bproj{C}\circ\nabla f^*$ is monotone.
\end{proof}

\begin{remark} \label{r:071210}
Let $y\in \RR^J=\dom f^*$. Then
$(f+\iota_{C})^*(y)=f^*(y)-\inf_{x\in C}[f(x)+f^*(y)-\scal{y}{x}].$
Since
$$f(x)+f^*(y)-\langle y,x\rangle=f(x)+f^*(\nabla f(\nabla f^*(y)))-\langle \nabla f(\nabla f^*(y)),x\rangle =
D(x,\nabla f^*(y)),$$
we have
$(f+\iota_{C})^*(y)=f^*(y)-\bD{C}(\nabla f^*(y)).$ Hence
$$(f+\iota_{C})^*=f^*-\bD{C}\circ\nabla f^*;$$ see also
Proposition~\ref{distance}.  If $f=\tfrac{1}{2}\|\cdot\|^2$, then
$$(f+\iota_{C})^*=\frac{1}{2}\|\cdot\|^2-\frac{1}{2}d_{C}^2$$
is the so-called \emph{Asplund function}, where
$d_{C}(y):=\inf\{\|y-x\|:\; x\in C\},$ $\forall y\in \RR^J$.
In this case, Lemma~\ref{conjugatesub} is classical; see
\cite[pages~262--264]{urruty3} or \cite{urruty4}.
\end{remark}

We also need the following result from \cite{solo}.

\begin{proposition}[Soloviov] \label{vlad}
Let $g:\RR^J\rightarrow\RX$ be lower semicontinuous,
and assume that $g^*$ is essentially smooth.
Then $g$ is convex.
\end{proposition}

Now we are ready for the main result of this section.
\begin{theorem}[Characterizations of $\bD{}$-Chebyshev sets]\label{mainpart}
The following are equivalent:
\begin{enumerate}
\item\label{whati} $C$ is convex;
\item\label{whatii} $C$ is \bDC, i.e., $\bproj{C}$ is single-valued on $U$;
\item\label{whatiii} $\bproj{C}$ is continuous on $U$;
\item \label{whativ} $\bD{C}\circ\nabla f^*$ is differentiable on $\RR^J$;
\item \label{whatv} $f+\iota_{C}$ is convex.
\end{enumerate}
When these equivalent conditions hold, we have
\begin{equation}\label{chenchen1}
\nabla (\bD{C}\circ \nabla f^*)=\nabla f^*-\bproj{C}\circ\nabla f^* \quad \mbox{ on $\RR^J$};
\end{equation}
consequently, $\bD{C}\circ \nabla f^*$ is continuously differentiable.

If, in addition, $f$ is twice continuously differentiable on $U$ and
$\Hess f(y)$ is positive definite $\forall y\in U$, then
\emph{\ref{whati}--\ref{whatv}} are equivalent to
\begin{enumerate}
\item[{\rm (vi)}] $\bD{C}$ is differentiable on $U$.
\end{enumerate}
In this case, we have
\begin{equation}\label{chenchen2}
\nabla\bD{C}(y)=\Hess f(y)[y-\bproj{C}(y)] \quad \mbox{ $\forall y\in U$};
\end{equation}
consequently, $\bD{C}$ is continuously differentiable.
\end{theorem}
\begin{proof}
\ref{whati}$\Rightarrow$\ref{whatii} is well known; see, e.g.,
\cite[Theorem~3.12]{Baus97}.
\ref{whatii}$\Rightarrow$\ref{whatiii} follows from
Theorem~\ref{banffday}(iii).
To see \ref{whatiii}$\Rightarrow$\ref{whativ}, we use Remark~\ref{r:071210}:
$$\bD{C}\circ\nabla f^*=f^*-(f+\iota_{C})^*.$$
Since $\bproj{C}$ is continuous on $U$ and $\nabla f^*:\RR^J\rightarrow U$,  $\partial (f+\iota_{C})^*$ is single-valued on $\RR^J$
by Lemma~\ref{conjugatesub}. Thus,
 $(f+\iota_{C})^*$ is differentiable on $\RR^J$.
Altogether, $\bD{C}\circ\nabla f^* = f^* - (f+\iota_C)^*$ is differentiable on
$\RR^J$.

When $f+\iota_{C}$ is convex, since $C\subset U$ we have that $\dom (f+\iota_{C})=C$ is convex, and this
shows \ref{whatv}$\Rightarrow$\ref{whati}.
We now prove \ref{whativ}$\Rightarrow$\ref{whatv} and assume \ref{whativ}.
Remark~\ref{r:071210} shows
\begin{equation}\label{chenchen3}
(f+\iota_{C})^*=f^*-\bD{C}\circ\nabla f^*,
\end{equation}
which implies that
\begin{equation}\label{betterform}
(f+\iota_{C})^* \mbox{ is differentiable on $\RR^J$.}
\end{equation}
Since $f+\iota_{C}$ is lower semicontinuous, it follows from Proposition~\ref{vlad} that $f+\iota_{C}$ is convex.

When equivalent conditions (i)-(v) hold,
(\ref{chenchen1}) follows from Lemma~\ref{conjugatesub} and (\ref{chenchen3}).
Since $\nabla f^*$ is continuous and $\bproj{C}$ is continuous by
\ref{whatiii}, we obtain that $\bD{C}\circ\nabla f^*$ is continuously
differentiable.
When $\Hess f(y)$ is positive definite $\forall y\in U$,
(ii)$\Leftrightarrow$(vi) by Corollary~\ref{singletonchar}.
Finally, (\ref{chenchen2}) follows from Theorem~\ref{complete}, i.e., (\ref{theset}). This finishes the proof.
\end{proof}

\section{Right Bregman Projections}\label{rightchar}

In this section, it will be convenient to write $D_{f}$ for
the {Bregman distance} associated with $f$ (see \eqref{eq:D}).
Correspondingly, we write
$\bproj{C}^{f}$, $\fproj{C}^{f}$ for the corresponding left and
right projection operators. While $D_{f}$ is convex in its
first argument, it is not necessarily so in its second argument.
The properties of $\fproj{C}^{f}$ can be studied by using
$\bproj{\nabla f(C)}^{f^*}$.

\begin{proposition}\label{different}
Let $f\in\Gamma$ be Legendre and $C\subset\inte\dom f$. Then
for the right Bregman nearest point projection, we have
\begin{equation}\label{left1.1}
\fproj{C}^{f}=\nabla f^*\circ \bproj{\nabla f(C)}^{f^*}\circ\nabla
f;
\end{equation}
or equivalently,
\begin{equation}\label{left1}
\bproj{\nabla f(C)}^{f^*}=\nabla f\circ \fproj{C}^{f}\circ\nabla
f^*.
\end{equation}
\end{proposition}
\begin{proof}
By \cite[Theorem~3.7(v)]{Baus97} (applied to $f^*$ rather than $f$),
$$D_{f^*}(x^*,y^*)=D_{f}(\nabla f^*(y^*),\nabla f^*(x^*)) \quad \forall x^*,y^*\in\inte\dom f^*.$$
For every $y^*\in\inte\dom f^*$, we thus have
\begin{align}
\bproj{\nabla f(C)}^{f^*}(y^*) & =\argmin_{x^*\in \nabla f(C)} D_{f^*}(x^*,y^*)=\argmin_{x^*\in \nabla f(C)} D_{f}(\nabla f^*(y^*),\nabla f^*(x^*))\label{manteo1}\\
& = \nabla f (\fproj{\nabla f^*(\nabla f(C))}^{f}(\nabla f^*(y^*)))=\nabla f(\fproj{C}^{f}(\nabla f^*(y^*)))\\
& =(\nabla f \circ \fproj{C}^{f}\circ\nabla
f^*)(y^*),\label{manteo3}
\end{align}
which gives (\ref{left1}). Finally, we see that (\ref{left1.1}) is
equivalent to (\ref{left1}) by using Fact~\ref{isom}.
\end{proof}

\begin{lemma}\label{preparation}
Let $f\in\Gamma$ be Legendre, let $C\subset\RR^J$ be such that
$\closu{C}\subset
\intdom f$, and assume that for every $y\in \intdom f$,
$\bproj{C}^{f}(y)\neq\varnothing$.
Then $C$ is closed.
\end{lemma}
\begin{proof}
Assume that $(c_{n})_{n=1}^{\infty}$ is a sequence in $C$,
and $c_{n}\to y$. We need to show that
$y\in C$. By assumption $y\in\closu{C}$, and
$y\in U$.
If $y\notin C$, then
\begin{equation}\label{strict}
D_{f}(c,y)=f(c)-f(y)-\langle\nabla f(y),c-y\rangle>0, \quad \forall c\in C,
\end{equation}
by, e.g., \cite[Theorem~3.7.(iv)]{Baus97}.
On the other hand, as $f$ is continuous on $U$,
$$
0\leq \bD{C}^{f}(y) \leq D_{f}(c_{n},y)= f(c_{n})-f(y)-\langle \nabla
f(y),c_{n}-y\rangle \to 0.
$$
Thus, $\bD{C}^{f}(y)=0$. Using (\ref{strict}), we see that this
 contradicts the assumption that $\bproj{C}^{f}(y)\neq \varnothing$.
\end{proof}

\begin{theorem}\label{termbegin}
Let $f\in\Gamma$ be Legendre, with full domain $\RR^J$,
and let $C\subset\RR^J$ be closed with $\closu (\nabla f(C))\subset\inte\dom f^*$. Assume that $\fproj{C}^{f}(y)$ is a
singleton for every $y\in\RR^J$. Then $\nabla f(C)$ is convex.
\end{theorem}
\begin{proof} We have $f^*$ is Legendre and $f^*$ is $1$-coercive.
By (\ref{left1}),
$\bproj{\nabla f(C)}^{f^*}(y)$ is single-valued for every
$y\in\inte\dom f^*$. As $\closu (\nabla f(C))\subset\inte\dom f^*$, Lemma~\ref{preparation} says that
the set $\nabla f(C)$ is closed.
Hence we apply Theorem~\ref{mainpart} to $f^*$ and $\nabla f(C)$, and we
obtain that $\nabla f(C)$ is convex.
\end{proof}

\begin{corollary}\label{termbegin1}
Let $f$ and $C$ satisfy \emph{\textbf{A1}--\textbf{A3}},
assume that $f$ has full domain, and
that $\fproj{C}^{f}(y)$ is a
singleton for every $y\in\RR^J$. Then $\nabla f(C)$ is convex.
\end{corollary}

The following example shows that even if $\fproj{C}^{f}(y)$ is a
singleton for every $y\in\inte\dom f$, the set $C$ may fail to be
convex. Thus, Theorem~\ref{mainpart} fails for the right Bregman
projection $\fproj{C}^f$. Note that
Theorem~\ref{termbegin} allows us to conclude
that $\nabla f(C)$ is convex rather than $C$.

\begin{example}
Consider the Legendre function $f:\RR^2\rightarrow\RR$ given by
$$f(x,y):=e^{x}+e^{y}\qquad \forall (x,y)\in\RR^2,$$
and its Fenchel conjugate
\begin{equation*}
f^*\colon \RR^2\rightarrow\RX \colon (x,y) \mapsto
\begin{cases}
x\ln x-x+y\ln y-y, &\text{if}\;\;x\geq 0, y\geq 0;\\
\pinf, & \text{otherwise}.
\end{cases}
\end{equation*}
Define a
compact convex set $$C:=\big[(0,0),(1,2)\big]=\{(\lambda,2\lambda):
0\leq\lambda\leq 1\}.$$
As $\nabla f(x,y)=(e^x,e^y)$ for every
$(x,y)\in \RR^2$, we see that
$$\nabla f(C)=\{(e^{\lambda},e^{2\lambda}):\ 0\leq \lambda\leq 1\}$$
is compact but clearly \emph{not convex}.

(i) In view of Theorem~\ref{termbegin} and
the lack of convexity of $\nabla f(C)$, there must exist
$(x,y)\in \RR^2$ such that $\fproj{C}^{f}(x,y)$ is \emph{multi-valued}.

(ii)
Since $\bproj{C}^{f}(x,y)$ is a singleton for every $(x,y)\in
\RR^2$, and since
$\fproj{\nabla f(C)}^{f^*}=\nabla f\circ \bproj{ C}^{f}\circ\nabla f^*$
by Proposition~\ref{different} (applied to $f^*$ and $\nabla f(C)$),
we deduce that
$\fproj{\nabla f(C)}^{f^*}$ is
single-valued on $\inte\dom f^*=\{(x,y):\; x>0,y>0\}.$
Therefore, the analogue of Theorem~\ref{mainpart}
for the right Bregman projection
Theorem~\ref{mainpart} fails even though $f^*$ is Legendre and
$1$-coercive.
\end{example}

\section*{Acknowledgments}
Heinz Bauschke was partially supported by the Natural Sciences and
Engineering Research Council of Canada and by the Canada Research Chair
Program.
Xianfu Wang was partially
supported by the Natural Sciences and Engineering Research Council
of Canada.
Jane Ye was partially
supported by the Natural Sciences and Engineering Research Council
of Canada.
Xiaoming Yuan was partially supported by the Pacific Institute for the
Mathematical Sciences, by the University of Victoria,
by the University of British Columbia Okanagan, and
by the National Science Foundation of China Grant~10701055.


\small

\begin{thebibliography}{99}
\bibitem{Edgar} E.\ Asplund, Cebysev sets in Hilbert space,
\emph{Trans.\ Amer.\ Math.\ Soc.}~144 (1969), 235--240.

\bibitem{Baus97}
H.\ H.\ Bauschke and J.\ M.\ Borwein,
Legendre functions and the method of random Bregman projections,
\emph{J.\ Convex Anal.}~4 (1997), 27--67.

\bibitem{Sico03}
H.\ H.\ Bauschke, J.\ M.\ Borwein, and P.\ L.\ Combettes,
Bregman monotone optimization algorithms,
\emph{SIAM J.\ Control Optim.}~42 (2003), 596--636.

\bibitem{noll}H. H. Bauschke, P. L. Combettes, and D. Noll,
Joint minimization with alternating Bregman proximity
operators, \emph{Pac. J. Optim.}~2  (2006), no. 3, 401--424.

\bibitem{BW}
H.\ Berens and U.\ Westphal,
Kodissipative metrische {P}rojektionen in normierten linearen
{R}\"aumen,
in \emph{Linear Spaces and Approximation} (P.~L. Butzer and B.~Sz.-Nagy,
eds.),
ISNM vol.~40 (1980), Birkh\"auser, 119--130.


\bibitem{jborwein} J.\ M.\ Borwein,
Proximity and Chebyshev sets, \emph{Optimization Letters},
to appear.

\bibitem{lewis} J.\ M.\ Borwein and A.\ S.\ Lewis,
\emph{Convex Analysis and Nonlinear Optimization},
second edition, Springer, New York, 2006.

\bibitem{ButIus}
D.\ Butnariu and A.\ N.\ Iusem,
\emph{Totally Convex Functions for Fixed Point Computation in Infinite
Dimensional Optimization},
Kluwer, Dordrecht, 2000.

\bibitem{CenZen}
Y.\ Censor and S.\ A.\ Zenios,
\emph{Parallel Optimization},
Oxford University Press, 1997.

\bibitem{Clarke}
F.\ H.\ Clarke,
\emph{Optimization and Nonsmooth Analysis},
SIAM, Philadelphia, 1990.

\bibitem{Yuri}
F.\ H.\ Clarke, Yu.\ S.\ Ledyaev, R.\ J.\ Stern, and P.\ R.\ Wolenski,
\emph{Nonsmooth Analysis and Control Theory},
Springer-Verlag, New York, 1998.

\bibitem{Frank1}
F.\ H.\ Clarke, R.\ J.\ Stern, and P.\ R.\ Wolenski,
Proximal smoothness and the lower-$C\sp 2$ property,
\emph{J. Convex Anal.}~2 (1995), 117--144.

\bibitem{Frank}
F.\ H.\ Clarke,
\emph{Optimization and Nonsmooth Analysis}, Wiley Interscience, New York, 1983.

\bibitem{deutsch}
F.\ Deutsch,
\emph{Best Approximation in Inner Product Spaces},
Springer-Verlag, New York, 2001.

\bibitem{urruty1}
J.-B. Hiriart-Urruty and C.\ Lemar\'echal,
\emph{Convex Analysis and Minimization Algorithms II},
Springer, New York, 1996.

\bibitem{urruty3} J.-B. Hiriart-Urruty,
Potpourri of conjectures and open questions in nonlinear analysis
and optimization,
\emph{SIAM Review}~49 (2007), 255--273.

\bibitem{urruty4} J.-B. Hiriart-Urruty, Ensembles de Tchebychev vs.
ensembles convexes: l'etat de la situation vu via l'analyse convexe
non lisse,
\emph{Ann.~Sci.~Math.~Qu\'ebec} 22 (1998), 47--62.

\bibitem{loewen} P.\ D.\ Loewen and X.\ Wang,
On the multiplicity of Dini subgradients
in separable spaces, \emph{Nonlinear Anal.}~58 (2004), 1--10.

\bibitem{mordukhovich} B.\ S.\ Mordukhovich, \emph{Variational Analysis and
Generalized Differentiation I}, Springer-Verlag, Berlin, 2006.

\bibitem{Rock70}
R.\ T.\ Rockafellar,
\emph{Convex Analysis},
Princeton University Press, Princeton, 1970.

\bibitem{Rock98}
R.\ T.\ Rockafellar and R.\ J-B Wets,
\emph{Variational Analysis},
Springer-Verlag, New York, 1998.

\bibitem{Simons}
S.\ Simons,
\emph{Minimax and Monotonicity},
Lecture Notes in Mathematics, vol.~1693, Springer-Verlag, 1998.

\bibitem{solo} V.\ Soloviov,
Duality for nonconvex optimization and its applications,
\emph{Anal. Math.}~19 (1993), 297--315.

\bibitem{WuYe} Z.\ Wu and J.\ J.\ Ye,
Equivalence among various derivatives and subdifferentials of the distance
function,
\emph{J.\ Math.\ Anal.\ Appl.}~282 (2003), 629--647.

\end{thebibliography}
\end{document}